\documentclass[12pt,a4paper]{amsart}
\usepackage[latin1]{inputenc}
\usepackage[english]{babel}
\usepackage{amsmath,amstext,amsthm,amsfonts,amssymb,amscd}
\usepackage[dvips]{graphicx}
\usepackage{hyperref}

\usepackage{cleveref}

\theoremstyle{plain}
\newtheorem{theorem}{Theorem}[subsection]
\newtheorem{proposition}{Proposition}[subsection]
\newtheorem{definition}{Definition}[subsection]
\newtheorem{remark}{Remark}[subsection]

\newtheorem{lemma}{Lemma}[subsection]

\newtheorem{theorema}{Theorem}

\begin{document}
	
	\title[Equilibrium Stability Open Zooming Systems]{Equilibrium Stability for Open Zooming Systems}
	
	
	
	\author[R. A. Bilbao]{Rafael A. Bilbao}
	\address{Rafael A. Bilbao, Universidad Pedag\'ogica y Tecnol\'ogica de Colombia, Avenida
		Central del Norte 39-115, Sede Central Tunja, Boyac\'a, 150003, Colombia. 
	}
	\email{rafael.alvarez@uptc.edu.co}

	
	
	\author[E. Santana]{Eduardo Santana}
	\address{Eduardo Santana, Universidade Federal de Alagoas, 57200-000 Penedo, Brazil}
	\email{jemsmath@gmail.com}
	


	\date{\today}



	
	
	\maketitle
	
	\begin{abstract} 
		We prove that for a wide family of open zooming systems and zooming potentials we have equilibrium stability, i.e., the equilibrium states depend continuously on the dynamics and the potential. We consider the open zooming systems with special holes and quite general contractions and zooming potentials with locally H\"older induced potential, which include the H\"older ones. We also prove stability for skew-products with the base being a zooming system like above. As a consequence of finiteness and stability, we obtain uniqueness of equilibrium state.
	\end{abstract}

	\bigskip


	
	\section{Introduction}
	The theory of nonuniformly expanding maps has been widely studied and the expanding measures are related. They are particular cases of what we call zooming systems, related to zooming measures introduced in \cite{Pi1}. Here, we deal with open zooming systems.
	
	Roughly speaking, a \textbf{\textit{zooming system}} is a map which extends the notion of non-uniform expansion,
	where we have a type of expansion obtained in the presence of hyperbolic times. The zooming times extend this notion beyond exponential contractions.
	In our context, a system is said to be \textbf{\textit{open}} when the phase space is not invariant. In other words, we begin with a closed map
	$f:M \to M$ and consider a Borel set $H \subset M$, in order to study the orbits with respect to $H$ (called the \textbf{\textit{hole}} of the system):
	if a point $x \in M$ is such that its orbit pass through $H$, we consider that the point $x$ escapes from the system. We observe that it reduces to a closed map when $H = \emptyset$.
	In particular, we study the set of points that never pass through $H$, called the \textbf{\textit{survivor set}}. For zooming systems, we consider holes such that the survivor set contains the zooming set. Also, once a reference measure $m$ is fixed, we study the \textbf{\textit{escape rate}} defined by:
	\[
	\displaystyle \mathcal{E}(f,m,H) = - \lim_{n \to \infty}\frac{1}{n}\log m\Big{(}\cap_{j=0}^{n-1} f^{-j}(M \backslash H)\Big{)}.
	\]
	The study of open dynamics began with Pianigiani and Yorke in the late 1970s (see \cite{PY}). The abstract concept of an open system is important because it leads immediately to the notion of a conditionally invariant measure and escape rate along with a host of detailed questions about how mass escapes or fails to escape from the system under time evolution.
	
	In order to study the escape rate, for example, Young towers are quite useful and the inducing schemes as well. For interval maps, inducing schemes are studied in \cite{DHL} and changes are made in \cite{BDM} to study inducing schemes for interval maps with holes. In the context of open dynamics, the inducing schemes need to be \textbf{\textit{adapted to the hole}}, as considered in \cite{DT}, for example (and referred to as \textbf{\textit{respecting the hole}}). 
	
	%
	
	The theory of equilibrium states in dynamical systems was first developed by Sinai, Ruelle and Bowen in the sixties and seventies. 
	It was based on applications of techniques of Statistical Mechanics to smooth dynamics. The classical theory of equilibrium states is developed for continuous maps. Given a continuous map $f: M \to M$ on a compact metric space $M$ and a continuous potential $\phi : M \to \mathbb{R}$, an  \textbf{\textit{equilibrium state}} is an 
	invariant measure that satisfies a variational principle, that is, a measure $\mu$ such that
	\begin{equation}\label{equilibrium}
		\displaystyle h_{\mu}(f) + \int \phi d\mu = \sup_{\eta \in \mathcal{M}_{f}(M)} \bigg{\{} h_{\eta}(f) + \int \phi d\eta \bigg{\}},
	\end{equation}
	where $\mathcal{M}_{f}(M)$ is the set of $f$-invariant probabilities on $M$ and $h_{\eta}(f)$ is the so-called metric entropy of $\eta$.
	
	For measurable maps we define equilibrium states as follows.
	Given a measurable map $f: M \to M$ on a compact metric space $M$ for which the set $\mathcal{M}_{f}(M)$ of invariant measures is non-empty (for example, the continuous maps) and a measurable potential $\phi : M \to \mathbb{R}$, we define the \textbf{\textit{pressure}} $P_{f}(\phi)$ as 
	\begin{equation}\label{pressure}
		\displaystyle P_{f}(\phi) : = \sup_{\eta \in \mathcal{M}_{f}(M)} \bigg{\{} h_{\eta}(f) + \int \phi d\eta \bigg{\}}
	\end{equation}
	and an \textbf{\textit{equilibrium state}} is an invariant measure that attains the supremum, and it is defined analogously to the case for continuous maps (see equation \ref{equilibrium}).
	
	In the context of uniform hyperbolicity (see \cite{B}), which includes uniformly expanding maps (see \cite{OV}), equilibrium states do exist and are unique if the potential is H\"older continuous
	and the map is transitive. In addition, the theory for finite shifts was developed and used to achieve the results for smooth dynamics.
	
	Beyond uniform hyperbolicity, the theory is still far from complete. It was studied by several authors, including Bruin, Keller, Demers, Li, Rivera-Letelier, Iommi and Todd 
	\cite{BDM, BK, BT,DT,IT1,IT2,LRL} for interval maps; Denker and Urbanski  \cite{DU} for rational maps; Leplaideur, Oliveira and Rios 
	\cite{LOR} for partially hyperbolic horseshoes; Buzzi, Sarig and Yuri \cite{BS,Y}, for countable Markov shifts and for piecewise expanding maps in one and higher dimensions. 
	For local diffeomorphisms with some kind of non-uniform expansion, there are results due to Oliveira \cite{O}; Arbieto, Matheus and Oliveira \cite{AMO};
	Varandas and Viana \cite{VV}, all of whom proved the existence and uniqueness of equilibrium states for potentials with low oscillation. Also, for this type of map,
	Ramos and Viana  \cite{RV}  proved it for the so-called  \textbf{\textit{hyperbolic potentials}}, which include these previous ones for the case of non-uniform expansion. The hyperbolicity of the potential is
	characterized by the fact that the pressure emanates from the hyperbolic region. In most of these studies previously cited, the maps do not have the presence of critical sets. In \cite{BDM} and \cite{DT}, for example, the authors develop results for open interval maps with critical sets, but not for hyperbolic potentials and, recently, Alves, Oliveira and Santana proved the existence of equilibrium states for hyperbolic potentials, possibly with the presence of a critical set (see \cite{AOS}). We will see that the potentials considered in \cite{LRL}, which allow critical sets, are included in the potentials considered in \cite{AOS}. Here, we give an example of a class of hyperbolic potentials. It includes the null potential for Viana maps. We stress that one of the consequences in this paper is the existence and uniqueness of measures of maximal entropy for Viana maps and we observe that in \cite{ALP} the authors show, in particular, the existence of at most countably many ergodic measures of maximal entropy. In \cite{PV} Pinheiro-Varandas obtain existence and uniqueness by using another technique for what they call \textbf{\textit{expanding potentials}}, which include the hyperbolic ones and are included in the \textbf{\textit{zooming potentials}}. We have that the class of hyperbolic potentials is equivalent to the class of continuous zooming potentials and that they include the null one. Similar results can be seen in \cite{BOS} for random systems.
	

	The present paper studies the equilibrium stability of open zooming systems. Once the finiteness of equilibrium states was obtained in \cite{S1}, the continuity of equilibrium states as studied in \cite{Ar} when we do not have uniqueness can be studied. In \cite{ARS} and \cite{BR} the authors study stability with uniqueness. Similar results of existence and uniqueness but  without stability can be seen in \cite{SSV}.
	
	We consider zooming systems with a quite general condition on the contractions and zooming potentials, proving the equilibrium stability in our Theorem \ref{A} for a family of such systems which has a common reference zooming measure. In Theorem \ref{thB} we also prove stability for skew-products with the base being a zooming system like above. In Theorem \ref{C} we use finiteness and stability to obtain uniqueness of equilibrium state.
	
	Our strategy is also similar to that used in \cite{ARS} where we prove that any accumulation point of a sequence of equilibrium states is also an equilibrium state for the limit system. We use a result in \cite{Ar} which deals with (semi)continuity of the pressure and the continuity of equilibrium states. The key point is the inequality between pressures. We also highlight the work \cite{FT} which also deals with this kind of problem. To prove stability for skew products we follow ideas of \cite{ARS}. Finally, to prove uniqueness, we show that the set of potentials with uniqueness is dense in the set of potentials with finiteness and by using stability we are able to prove uniqueness in general.

	\bigskip

	\paragraph{\bf Acknowledgement}{The authors would like to thank Jos\'e F. Alves for fruitful conversations to improve this work.}

	\section{Setup and Main Results}\label{Setup}
	
	In this section we give some definitions and state our main results. 
	
	\subsection{Zooming Sets and Measures}
	
	For differentiable dynamical systems, hyperbolic times are a powerful tool to obtain a type of expansion in the context of non-uniform expansion. As we can find in \cite{AOS}, it may be generalized for systems considered in a metric space, also with exponential contractions. The zooming times generalizes it beyond the exponential
	context. Details can be seen in \cite{Pi1}.
	
	Let $f : M \to M$ be a measurable map defined on a connected, compact, separable metric space $M$.
	
	\begin{definition}
		(Zooming contractions). A \textbf{\textit{zooming contraction}} is a sequence of functions $\alpha_{n}: [0,+\infty) \to [0,+\infty)$ such that
		
		\begin{itemize}
			\item $\alpha_{n}(r) < r, \text{for all} \, \, n \in \mathbb{N}, \text{for all} \, \, r>0.$
			
			\item $\alpha_{n}(r)<\alpha_{n}(s), \,\, if \,\, 0<r<s, \text{for all} \, \, n \in \mathbb{N}$.
			
			\item $\alpha_{m} \circ \alpha_{n}(r) \leq \alpha_{m+n}(r), \text{for all} \, \, r>0, \text{for all} \, \, m,n \in \mathbb{N}$.
			
			\item $\displaystyle \sup_{r \in (0,1)} \sum_{n=1}^{\infty}\alpha_{n}(r) < \infty$.
		\end{itemize}
		
	\end{definition}
	
	As defined in \cite{PV}, we call the contraction $(\alpha_{n})_{n}$ \textbf{\textit{exponential}} if $\alpha_{n}(r) = e^{-\lambda n} r$ for some $\lambda > 0$ and \textbf{\textit{Lipschitz}} if $\alpha_{n}(r) = a_{n} r$ with $0 \leq a_{n} < 1, a_{m}a_{n} \leq a_{m+n}$ and $\sum_{n=1}^{\infty} a_{n} < \infty$. In particular, every exponential contraction is Lipschitz. We can also have the example with $a_{n} = (n+b)^{-a}, a > 1, b>0$.

	\begin{definition}\label{times}
		(Zooming times). Let $(\alpha_{n})_{n}$ be a zooming contraction and $\delta>0$. We say that $n \in \mathbb{N}$ is an $(\alpha,\delta)$\textbf{\textit{-zooming time}} for $p \in X$ if there exists 
		a neighbourhood $V_{n}(p)$ of $p$ such that
		
		\begin{itemize}
			\item $f^{n}$ sends $\overline{V_{n}(p)}$ homeomorphically onto $\overline{B_{\delta}(f^{n}(p))}$;
			
			\item $d(f^{j}(x),f^{j}(y)) \leq \alpha_{n - j}(d(f^{n}(x),f^{n}(y)))$ for every $x,y \in V_{n}(p)$ and every $0 \leq j < n$.
		\end{itemize}
		
		We call $B_{\delta}(f^{n}(p))$ a \textbf{\textit{zooming ball}} and $V_{n}(p)$ a  \textbf{\textit{zooming pre-ball}}.
	\end{definition}
	
	We denote by $Z_{n}(\alpha,\delta,f)$ the set of points in $M$ for which $n$ is an $(\alpha, \delta)$- zooming time.
	
	\begin{definition}
		(Zooming measure) A $f$-non-singular finite measure $\mu$ defined on the Borel sets of M is called a \textbf{\textit{weak zooming measure}} if $\mu$ almost every point has 
		infinitely many $(\alpha, \delta)$-zooming times. A weak zooming measure is called a \textbf{\textit{zooming measure}} if
	\end{definition}
	\[
	\displaystyle \limsup_{n \to \infty} \frac{1}{n} \{1 \leq j \leq n \mid x \in Z_{j}(\alpha,\delta,f)\} > 0,
	\]
	$\mu$ almost every $x \in M$.
	
	\begin{definition}
		(Zooming set) A forward invariant set $\Lambda=\Lambda(\alpha,\delta,\mu) \subset M$ is called a \textbf{\textit{zooming set}} if the above inequality holds for every $x \in \Lambda$. 
	\end{definition}
	\begin{remark}
		We stress that the definition of the zooming set $\Lambda$ depends on the contraction $\alpha$ and the parameter $\delta$. Moreover, every positively invariant subset $\Lambda_{0} \subset \Lambda$ is also a zooming set if $\mu(\Lambda_{0}) = \mu(\Lambda)$. 
	\end{remark}
	
	\begin{definition}
		(Bounded distortion) Given a measure $\mu$ with a jacobian $J_{\mu}f$, we say that the measure has \textbf{\textit{bounded distortion}} if there exists $\rho > 0$ such that
		\[
		\bigg{|}\log \frac{J_{\mu}f(y)}{J_{\mu}f(z)} \bigg{|} \leq \rho d(f^{n}(y),f^{n}(z)),
		\]
		for every $y,z \in V_{n}(x)$, $\mu$-almost everywhere $x \in M$, for every zooming time $n$ of $x$.
	\end{definition}
	
	The map $f$ with an associated zooming measure is called a \textbf{\textit{zooming system}}.
	
	\bigskip
	
	\subsection{Markov structure with special holes} The following result guarantees the existence of Markov structures adapted to holes of a special type. See details in \cite{S1}[Theorem A].
	
	\begin{theorem}
		\label{structure}
		Let $(f,M,\mu,\Lambda)$ be a zooming system. Let $r_{0} > 0$ and let $\mathcal{A} = \{B_{r}(p_{i}) \mid i=1,2,\dots,k , r < r_{0}\}$ be a finite open cover of $M$ such that $B_{r}(p_{i}) \cap B_{r/2}(p_{j}) = \emptyset, \text{for all} \, \, i,j \leq k, j\neq i$. We assume that $H \subset M$ satisfies either $H = \emptyset$ or $H$ is a special type of open set chosen such that   
		\[
		\displaystyle \cup_{j=1}^{t} B_{\frac{r_{k_{j}}}{2}}(p_{k_{j}}) \subset H \subset \cup_{j=1}^{t} B_{r_{k_{j}}}(p_{k_{j}}),
		\]
		for some $k_{1},\dots,k_{t} \leq k$. Then if $r_{0}$ is sufficiently small, depending only on $(f,M,\mu,\Lambda)$, there exists a finite Markov structure adapted to the hole $H$.
	\end{theorem}

	\subsection{Pressure, zooming potentials and equilibrium states} For measurable maps we recall the definition of an equilibrium state, given in the Introduction. Given a measurable map $f: M \to M$ on a compact metric space $M$ for which the set $\mathcal{M}_{f}(M)$ of $f$-invariant measures is non-empty (for example, the continuous maps) and a measurable potential $\phi : M \to \mathbb{R}$, we define the \textbf{\textit{pressure}} $P_{f}(\phi)$ as in equation \ref{pressure} and an  \textbf{\textit{equilibrium state}} is an invariant measure that attains the supremum as in equation \ref{equilibrium}.
	
	Denote by $\mathcal{Z}(\Lambda)$ the set of invariant zooming measures supported on $\Lambda$ (in [\cite{Pi1}, Theorem C] it is proved that this set is nonempty under the same hypothesis as ours for the map, that is, a zooming measure with bounded distortion). 
	
	We define a \textbf{\textit{zooming potential}} as a measurable potential $\phi:M \to \mathbb{R}$ such that
	\[
	\displaystyle  \sup_{\eta \in \mathcal{Z}(\Lambda)^{c}} \bigg{\{} h_{\eta}(f) + \int \phi d\eta \bigg{\}} < \sup_{\eta \in \mathcal{Z}(\Lambda)} \bigg{\{} h_{\eta}(f) + \int \phi d\eta \bigg{\}}.
	\]
	It is analogous to the definition of expanding potential that can be found in \cite{PV}.
	
	Denote by $h(f)$ the pressure of the potential $\phi \equiv 0$ (or topological entropy), which we call simply \textbf{\textit{entropy}} of $f$, that is,
	\[
	\displaystyle h(f) := P_{f}(0)  = \sup_{\eta \in \mathcal{M}_{f}(M)} \bigg{\{} h_{\eta}(f) \bigg{\}}.
	\]

	For open systems, we consider another type of pressure and equilibrium states as follows.
	
	\begin{definition}(Open pressure and open equilibrium states)
		Given an open system $(f,M,H)$ with a hole $H$, we define the \textbf{\textit{open pressure}} as
		\begin{equation}\label{open pressure}
			\displaystyle P_{f,H}(\phi) : = \sup_{\eta \in \mathcal{M}_{f}(M,H)} \bigg{\{} h_{\eta}(f) + \int \phi d\eta \bigg{\}},
		\end{equation}
		where $\mathcal{M}_{f}(M,H)$ is the set of $f$-invariant measures $\eta$ such that $\eta(H) = 0$. Moreover, we define an \textbf{\textit{open equilibrium state}} as an invariant measure $\mu \in \mathcal{M}_{f}(M,H)$ such that
		\begin{equation}\label{open equilibrium}
			\displaystyle h_{\mu}(f) + \int \phi d\mu = \sup_{\eta \in \mathcal{M}_{f}(M,H)} \bigg{\{} h_{\eta}(f) + \int \phi d\eta \bigg{\}},
		\end{equation} 	
	\end{definition}

	\begin{definition}
		(Backward separated map) We say that a map $f:M \to M$ is \textbf{\textit{backward separated}} if for every finite set $F \subset M$ we have
		\[
		\displaystyle d\big{(}F, \cup_{j=1}^{n} f^{-j}(F) \backslash F \big{)} > 0, \text{for all} \, \,\, n \geq 1.
		\]
	\end{definition}
	Observe that if $f$ is such that $\sup\{\# f^{-1} (x) \mid x \in M \}< \infty$, then $f$ is backward separated.

	\begin{definition}
		(Induced potential) Given an inducing scheme $(F,\mathcal{P})$ and a potential $\phi:M \to \mathbb{R}$ we define the  \textbf{\textit{induced potential}} as 
		\[
		\overline{\phi}(x) =\sum_{j=0}^{R(x)-1} \phi(f^{j}(x)).
		\]
	\end{definition}
	
	\begin{definition}
		(Locally H\"older potential) Let $\Sigma$ be the space of symbols associated to the inducing scheme $(F,\mathcal{P})$. Given a potential $\Phi : \Sigma \to \mathbb{R}$, we say that $\Phi$ is \emph{locally H\"older} if there exist $A > 0$ and 
		$\theta \in (0,1)$ such that for all $n \in \mathbb{N}$ 
		\[
		V_{n}(\Phi) : = \sup\left\{|\Phi(x) - \Phi(y)| \colon  x, y \in C_{n}\right\} \leq A\theta^{n}.
		\]
	\end{definition}

	\begin{theorem}[\cite{S1}, Theorem B]
		\label{deterministic}
		Given a measurable open zooming system $f:M \to M$ which is backward separated, if the contraction $(\alpha_{n})_{n}$ satisfies $\alpha_{n}(r) \leq ar$ for some $a \in (0,1)$, every $n \in \mathbb{N}$ and every $r \in [0,+\infty)$ (Lipschitz, for example), with zooming set  $\Lambda$ (previously fixed) and hole $H$ given by Theorem \ref{structure}. 
		
		\begin{itemize}
			
			\item If $\phi:M \to \mathbb{R}$ is zooming potential with  finite pressure $P_{f}(\phi)$ and locally H\"older induced potential ($\phi$ H\"older, for example), then there are \textbf{finitely many}  ergodic equilibrium states and they are zooming measures.  
			
			\item If the zooming set $\Lambda$ is not dense in $M$, we can choose the hole such that $\Lambda \cap H = \emptyset$ and obtain the equilibrium states giving full mass to the survivor set $M^{\infty}$. Then, in this case, the equilibrium states are open.
			
		\end{itemize}
		
	\end{theorem}

	\subsection{Equilibrium stability} Let $\mathcal{CM}$ be the family of continuous maps $f: M \to M$ and $\mathcal{CP}$ the family of continuous potentials $\phi : M \to \mathbb{R}$. We say that a family $\mathcal{F} \subset \mathcal{CM} \times \mathcal{CP}$ for which every pair $(f,\phi) \in \mathcal{F}$ has existence of equilibrium states is  \textbf{\textit{equilibrium stable}} if we have continuity of the equilibrium states. It means that if the pair $(f_{n},\phi_{n}), n \geq 1$ has the measure $\mu_{n}$ as an equilibrium state and $f_{n} \to f_{0}$ and $\phi_{n} \to \phi_{0}$, then every accumulation point $\mu_{0}$ of $\mu_{n}$ is an equilibrium state for the pair $(f_{0},\phi_{0})$. The topology on $\mathcal{CM}$ is the $C^{0}$ one, on $\mathcal{CP}$ is also the $C^{0}$ topology and on the measures is the weak-$*$ topology. The topology on $\mathcal{CM} \times \mathcal{CP}$ is the product topology. By Theorem \ref{deterministic} we have that the following family has finiteness of equilibrium states:
	\[
	\mathcal{FZ} = \{(f,\phi) \mid f,\phi \,\, \text{are both zooming  and} \,\, \phi \,\, \text{has induced potential locally H\"older} \}.
	\]

	\begin{theorema}
		\label{A}
		$\mathcal{FZ}$ is equilibrium stable. 
	\end{theorema}
	
	\subsection{Skew products}\label{skew} Let $F:M\times N \rightarrow M\times N$ be a continuous map,  with  $F(x,y)=(f(x), g(x,y))$, where $f:M\longrightarrow M$ is a zooming map with contractions given in Theorem \ref{deterministic} and  $N$ is a Polish compact metric space with a distance $d_N$, $g:M\times N \rightarrow N$ is a continuous map, which is a uniform contraction on $N$, i.e. there exists $0< \lambda < 1$ such that for all $x\in M$ and all $y_1, y_2 \in N$ we have 
	\begin{equation}
		\label{equacontracao}
		d_{N}(g(x,y_1),g(x,y_2)) \le \lambda d_{N}(y_1, y_2),
	\end{equation}
	
	
	In this work, we will consider two types of skew-products
	\begin{itemize}
		\label{conditionong}
		\item [i.)] $g$ satisfies (\ref{equacontracao}).
		\item [ii.)] $g$ satisfies (\ref{equacontracao}) and there exists $y_0\in N$ such that $g(x, y_0)= y_0$ for all $x\in M$.
	\end{itemize}
	We will denote this family of skew-products with this property as $\mathcal{S}$, i.e.
	$$
	\mathcal{S}=\{F:M\times N \rightarrow M\times N, F(x,y)=(f(x), g(x,y)) : \ (f, \varphi)\in \mathcal{F} \ \text{and} \ g \ \text{satisfies} \ (\ref{equacontracao}) \}.
	$$
	Given $F\in \mathcal{S}$, we say that a continuous potential $\phi:M \times N\rightarrow \mathbb{R}$ is zooming for $F$ if the topological pressure of the system $(F,\phi)$ is equal to the relative pressure on the set $\Lambda(\alpha, \delta, \mu)(f)\times N$, 
	i.e., a measurable function $\phi:M\times N\rightarrow \mathbb{R}$ is a \textbf{\textit{zooming potential}} for $F$ if
	$$
	\sup_{\eta\in \mathcal{Z}(({\Lambda(\alpha,\delta, \mu)(f)})^{c}\times N)}\left\{ h_{\eta}(F) + \int \phi \ d\eta \right\} < \sup_{\eta\in \mathcal{Z}({\Lambda(\alpha,\delta, \mu)(f)}\times N)}\left\{ h_{\eta}(F) + \int \phi \ d\eta \right\}.
	$$	
	We will denote this set of zooming potentials for $F$ as:
	
	$$
	(\mathcal{FZ})'= \{(F, \phi)\in \mathcal{S}\times C^{0}(M\times N, \mathbb{R}): \ \phi \ \text{is zooming for} \ F \}.
	$$
	Similarly to Theorem \ref{A}, we will have the following result for skew-products.
	\begin{theorema}
		\label{thB}
		$(\mathcal{FZ})'$ is equilibrium stable.
	\end{theorema}
	
	
	As a consequence of finiteness and stability, we obtain the following theorem.
	
	\begin{theorema}\label{C}
		Under the hypothesis of Theorem \ref{deterministic} and with stability by Theorem \ref{A}, we obtain that there exist finitely many equilibrium states for all the potentials and also  uniqueness of equilibrium state.
	\end{theorema}
	
	We also obtain uniqueness for the skew products in subsection \ref{skew}.

	\section{Proof of Theorem \ref{A}}

	\begin{proof}
		We take $(f_{n},\phi_{n}), n\geq 1$ a sequence in $\mathcal{FZ}$ and $\mu_{n}$ an ergodic equilibrium state for the pair $(f_{n},\phi_{n}), n \geq 1$. We assume that $(f_{n},\phi_{n}) \to (f,\phi)$ which has existence of equilibrium state. Denote by $\mathcal{Z}_{n} = \mathcal{Z}(\Lambda_{n})$ the set of invariant zooming measures with respect to $f_{n}$ and $\mathcal{Z} = \mathcal{Z}(\Lambda)$ the set of invariant zooming measures with respect to $f$. Then,
		\[
		\displaystyle P_{f_{n}}(\phi_{n}) : = \sup_{\eta \in \mathcal{Z}_{n}} \bigg{\{} h_{\eta}(f_{n}) + \int \phi_{n} d\eta \bigg{\}} = h_{\mu_{n}}(f_{n}) + \int \phi_{n} d\mu_{n}, n \geq 1.
		\]
		Let $\mu_{0}$ be an accumulation point of the sequence $\mu_{n}$. We  show that $\mu_{0}$ is an invariant probability for $f$ (as in \cite{ARS}): Since each $\mu_{n}$ is an $f_{n}$-invariant measure, for any continuous function $\varphi: M \to \mathbb{R}$ we have  
		\[
		\int \varphi \circ f_{n} d\mu_{n} = \int \varphi d\mu_{n} \to \int \varphi d\mu_{0} \,\, \text{when} \,\, n \to \infty.
		\]
		Hence, to verify the $f$-invariance of $\mu_{0}$ it suffices to prove that
		\[
		\int \varphi \circ f_{n} d\mu_{n} \to \int \varphi \circ f d\mu_{0} \,\, \text{when} \,\, n \to \infty.
		\]
		For each $n \in \mathbb{N}$ we may write the inequality
		\[
		\Bigg{|} \int \varphi \circ f_{n} d\mu_{n} - \int \varphi \circ f d\mu_{0}  \Bigg{|} \leq \Bigg{|} \int \varphi \circ f_{n} d\mu_{n} - \int \varphi \circ f d\mu_{n}  \Bigg{|} + 
		\]
		\[
		+ \Bigg{|} \int \varphi \circ f d\mu_{n} - \int \varphi \circ f d\mu_{0}  \Bigg{|}.
		\]
		Combining the convergence of $f_{n}$ to $f$ and the fact that $\mu_{0}$ is an accumulation point of the
		sequence $\{\mu_{n}\}_{n}$, we have that each term in the sum above is close to zero for $n$ sufficiently
		large. This implies that $\mu_{0}$ is $f$-invariant.
		
		We now need to show that $\mu_{0}$ is an equilibrium state for the pair $(f,\phi)$ and, by arbitrariness of the sequence, we obtain the stability. 
		
		Since the equilibrium states do not change if we sum a constant to the potential, we know that we also have  $\mu_{n}$ as an equilibrium state for the system $(f_{n},\phi_{n} - P_{f_{n}}(\phi_{n}))$ and $P_{f_{n}}(\phi_{n} - P_{f_{n}}(\phi_{n})) = 0$. Given $\mu$ an equilibrium state for the system $(f_{0}, \phi_{0})$ we have $\mu$ an equilibrium state for the system $(f_{0},\phi_{0} - P_{f_{0}}(\phi_{0}))$ and $P_{f_{0}}(\phi_{0} - P_{f_{0}}(\phi_{0})) = 0$.
		
		As in \cite{Ar}[Theorem 11] we have that the existence of a generating partition gives the inequality in the following lemma.
		\begin{lemma}\label{generating}
			There exists a generating partition $\mathcal{P}$ such that $\mu_{0}(\partial \mathcal{P}) = 0$ and we have that
			\[
			\limsup_{n \to \infty}h_{\mu_{n}}(f_{n}) \leq h_{\mu_{0}}(f,\mathcal{P}). 
			\] 
		\end{lemma}
		\begin{proof}
			Since every equilibrium state is a zooming measure, we may proceed similarly to what is made in \cite{ARS} to prove the existence of $\mathcal{P}$.
			
			Let $\delta > 0$ be as in Definition \ref{times} and consider $\mathcal{P}$ a finite partition of $M$ with a diameter smaller that $\delta/2$ and $\mu_{0}(\partial \mathcal{P}) = 0$. For all $n \geq 0$ and for all $m \geq 1$, define the partition $\mathcal{P}_{m}^{n}$ by
			\[
			\mathcal{P}_{m}^{n}: = \{P_{m}^{n} = P_{i_{0}} \cap \dots \cap f_{n}^{-(m-1)}(P_{i_{m-1}}) \mid P_{i_{j}} \in \mathcal{P}, 0 \leq j \leq m-1 \}.
			\] 
			Given $x \in M$, define also $P_{m}^{n}(x)$ as the element $P_{m}^{n} \in \mathcal{P}_{m}^{n}$ such that $x \in P_{m}^{n}$. Note that, by definition, the sequence $\{P_{m}^{n}(x)\}_{m \geq 1}$ is non-increasing in $m$, meaning that
			\[
			P_{m+1}^{n}(x) \subset P_{m}^{n}(x), \,\, \text{for all} \,\, m \geq 1.
			\]
			Let $\Lambda_{n}$ be the zooming set of $f_{n}$. Given $x \in \Lambda_{n}$, let $m = m(n,x)$ be a zooming time for $x$. Let $\alpha_{m}^{n}(r)$ a sequence of contractions of the map $f_{n}$ as in Theorem \ref{deterministic}. We have that $\text{diam}(P_{m}^{n}(x)) \leq \alpha_{m}^{n}(2\delta)$. Assuming $2\delta < 1$, since every point $x \in \Lambda_{n}$ has infinitely many zooming times, we conclude that $\text{diam}(P_{m}^{n}(x)) \to 0$ when $m \to \infty$ for all $x \in \Lambda_{n}$ and for all $n \geq 0$, because we have for every $n \geq 0$ that $\displaystyle \sup_{r \in (0,1)} \sum_{m=1}^{\infty}\alpha_{m}^{n}(r) < \infty$ and $\alpha_{m}^{n}(r) \to 0$. Now by using \cite{Ar}[Theorem 11] we get the inequality and it proves the lemma.
		\end{proof}
		
		We now proceed to prove a type of semicontinuity of the pressure. We consider the sequence of C\`esaro's average
		\[
		\eta_{k}^{n} = \frac{1}{k} \sum_{i=0}^{k-1} f_{n*}^{i}\mu.
		\]
		There exist $\{k_{j}\}_{j \geq 1}$ and an $f_{n}$-invariant  measure $\eta_{n}$ such that $\eta_{k_{j}}^{n} \to \eta_{n}$ when $j \to \infty$. We now use the Brin-Katok's local formula of entropy to obtain the inequality above. The formula says that for an ergodic measure $\eta$ it holds that (see \cite{OV}) for $\eta$ almost every $x \in M$ that
		\[
		h_{\eta}(f) = \sup_{\epsilon > 0} h_{\eta}^{+}(f,\epsilon,x), \,\, \text{where} \,\, h_{\eta}^{+}(f,\epsilon,x): = \limsup_{n \to \infty} -\frac{1}{n}\log\eta(B_{f}(x,n,\epsilon)),
		\]
		and where
		\[
		B_{f}(x,n,\epsilon) = \{y \in M \mid d(f^{i}(x),f^{i}(y)) < \epsilon; 0 \leq i \leq n-1\} = \bigcap_{i = 0}^{n-1} f^{-i}(B_{\epsilon}(f^{i}(x))).
		\]
		We can see that $\epsilon'< \epsilon$ implies $B_{f}(x,n,\epsilon') \subset B_{f}(x,n,\epsilon)$ and $\partial B_{f}(x,n,\epsilon') \cap \partial B_{f}(x,n,\epsilon) = \emptyset$.
		\begin{lemma}\label{continuity}
			We have $\mu(B_{f}(x,n,\epsilon)) = \eta_{p}(B_{f}(x,n,\epsilon))$ for every $\epsilon > 0$ except possibly for countably many parameters.
		\end{lemma}
		\begin{proof}
			The measure $\eta_{p}$ gives null mass for the set $\partial B_{f}(x,n,\epsilon)$ for every $\epsilon > 0$ except possibly for countably many parameters because $\eta_{p}(\partial (B_{f}(x,n,\epsilon)) = 0$ except possibly at most countably many parameters $\epsilon$, since its well known that a sum of uncountably many positive numbers always diverges. It implies that $B_{f}(x,n,\epsilon)$ is a continuity set for $\eta_{p}$ and the following holds
			\[
			\mu(B_{f}(x,n,\epsilon)) = \eta_{k_{j}}^{p}(B_{f}(x,n,\epsilon)) \to \eta_{p}(B_{f}(x,n,\epsilon)).
			\]
		\end{proof}
		
		The following lemma guarantees a semicontinuity of the pressure.
		\begin{lemma}\label{BK}
			There exists an $f_{n}$-ergodic measure $\nu_{n}$ such that	\[
			h_{\mu}(f) \leq \limsup_{n \to \infty} h_{\nu_{n}}(f_{n}) \,\, \text{and} \,\, \limsup_{n \to \infty}P_{f_{n}}(\phi_{n}) \geq P_{f}(\phi).
			\]
		\end{lemma}
		\begin{proof}
			In order to obtain the measure $\nu_{n}$, 
			Since $f_{n} \to f$ uniformly, there exists $p_{n} \in \mathbb{N}$ such that for $p \geq p_{n}$ it implies $B_{f_{p}}(x,n,\epsilon/2) \subset B_{f}(x,n,\epsilon)$. Hence, we obtain for  $\mu$-almost everywhere $x \in M$ and using Lemma \ref{continuity} that
			\[
			h_{\mu}(f) = \sup_{\epsilon > 0} \limsup_{n \to \infty} -\frac{1}{n}\log\mu(B_{f}(x,n,\epsilon)) = \sup_{\epsilon > 0} \limsup_{n \to \infty} -\frac{1}{n}\log\eta_{p}(B_{f}(x,n,\epsilon)). 
			\]
			By the Ergodic Decomposition Theorem, there exists an ergodic component $\nu_{p}$ of $\eta_{p}$ such that $\eta_{p}(B_{f}(x,n,\epsilon)) \geq \nu_{p}(B_{f}(x,n,\epsilon))$. Then, we get 
			\[ 
			\sup_{\epsilon > 0} \limsup_{n \to \infty} -\frac{1}{n}\log\eta_{p}(B_{f}(x,n,\epsilon)) \leq
			\sup_{\epsilon > 0} \limsup_{n \to \infty} -\frac{1}{n}\log\nu_{p}(B_{f}(x,n,\epsilon)) \leq 
			\]
			\[
			\sup_{\epsilon > 0} \limsup_{n \to \infty} -\frac{1}{n}\log\nu_{p}(B_{f_{p}}(x,n,\epsilon/2)) \leq \limsup_{p \to \infty}\sup_{\epsilon > 0} h_{\nu_{p}}^{+}(f_{p},\epsilon/2,x) \leq \limsup_{p \to \infty} h_{\nu_{p}}(f_{p}), p \geq p_{n}.
			\]
			Fix $\epsilon > 0$ small enough. Since $\mu$ is an equilibrium state for $f$, we have for $\phi < 0, \phi_{n} < 0, n \geq n_{0}$ for some $n_{0}$ and choose $\alpha < 0 < \beta < 1$. Since $\alpha \sin^{2}(\beta) < 0$ and $\sin^{2}(\beta) = 1- \cos^{2}(\beta)$, for every $\alpha,\beta$ as above, we have
			\begin{align*}
				& P_{f}(\phi)  = h_{\mu}(f) + \int \phi d\mu \leq \limsup_{n \to \infty} ( h_{\nu_{n}}(f_{n}) + \int \phi d\mu) \leq \limsup_{n \to \infty} h_{\nu_{n}}(f_{n}) \leq \\
				&\limsup_{n \to \infty} (h_{\nu_{n}}(f_{n}) + (\alpha\sin^{2}(\beta))\int  (-\phi_{n}/2(1 -\cos(\beta))) d\nu_{n} - (\alpha\sin^{2}(\beta)) P_{f_{n}}(- \phi_{n}/2(1 -\cos(\beta))) = \\
				&\limsup_{n \to \infty} (h_{\nu_{n}}(f_{n}) + \int  (-\alpha 
				(1 + 
				\cos(\beta))\phi_{n}/2) d\nu_{n} - (\alpha\sin^{2}(\beta)) P_{f_{n}}(- \phi_{n}/2(1 -\cos(\beta)))) \leq \\
				&\limsup_{n \to \infty} (h_{\nu_{n}}(f_{n}) + \int  (-\alpha 
				(1 + 
				\cos(\beta))\phi_{n}/2) d\nu_{n} - (\alpha\sin^{2}(\beta)) P_{f_{n}}(- \phi_{n}/2(1 -\cos(\beta - \epsilon)))) \implies \\ 
				& P_{f}(\phi) \leq \limsup_{n \to \infty} (P_{f_{n}} (-\alpha 
				(1 + 
				\cos(\beta))\phi_{n}/2) - (\alpha\sin^{2}(\beta)) P_{f_{n}}(- \phi_{n}/2(1 -\cos(\beta - \epsilon)))) \implies_{(
					\alpha = -1)} \\
				& P_{f}(\phi) \leq \limsup_{n \to \infty} (P_{f_{n}}((1 + \cos(\beta))\phi_{n}/2 ) + \sin^{2}(\beta) P_{f_{n}}(- \phi_{n}/2(1 -\cos(\beta - \epsilon)))) \implies_{(\beta \to 0)}    \\
				& P_{f}(\phi) \leq \limsup_{n \to \infty} (P_{f_{n}}(\phi_{n}) + 0 P_{f_{n}}(-\phi_{n}/2(1 -\cos(- \epsilon))) \implies \\
				& P_{f}(\phi) \leq \limsup_{n \to \infty} P_{f_{n}}(\phi_{n}) 
			\end{align*}
			The Lemma is proved.
		\end{proof}
		
		Consider the sequence of functions $\varphi_{n}$ defined as 
		\[
		\varphi_{n}: = \phi_{n} - P_{f}(\phi_{n}) \implies P_{f_{n}}(\varphi_{n}) = P_{f_{n}}(\phi_{n}) - P_{f}(\phi_{n}).
		\]

		Lemma \ref{generating} gives a generating partition $\mathcal{P}$ and Lemma \ref{BK} imply that 
		\[
		h_{\mu_{0}}(f) +  \int (\phi - P_{f}(\phi)) d\mu_{0} \geq h_{\mu_{0}}(f,\mathcal{P}) + \limsup_{n \to \infty} \int  \varphi_{n}d\mu_{n} \geq
		\limsup_{n \to \infty} (h_{\mu_{n}}(f_{n}) +  \int  \varphi_{n}d\mu_{n}) =
		\]
		\[  
		\limsup_{n \to \infty} P_{f_{n}}(\varphi_{n}) = \limsup_{n \to \infty}(P_{f_{n}}(\phi_{n}) - P_{f}(\phi_{n})) = \limsup_{n \to \infty}(P_{f_{n}}(\phi_{n})) - P_{f}(\phi) \geq 0.
		\]
		Hence, we get $0 = P_{f}(\phi - P_{f}(\phi)) \geq  h_{\mu_{0}}(f) +  \int (\phi - P_{f}(\phi)) d\mu_{0} \geq 0$. It implies that $h_{\mu_{0}}(f) +  \int (\phi - P_{f}(\phi)) d\mu_{0} = 0 = P_{f}(\phi - P_{f}(\phi))$ and so $\mu_{0}$ is an equilibrium state for the pair $(f, \phi - P_{f}(\phi))$, that is, for the pair $(f, \phi)$. It means that $\mathcal{FZ}$ is equilibrium stable. Theorem \ref{A} is proved.
	\end{proof}

	\section{Proof of Theorem \ref{thB}} \label{B}

	The proofs of the following proposition and lemmas can be found in \cite{ARS} where they are proven for non-uniformly expanding maps with no critical points, that is, zooming systems with exponential contractions. First, let us begin by defining when two potentials are homologous. We say that
	two potentials $\widetilde{\phi},\phi:M\times N \rightarrow \mathbb{R}$ are homologous if there is a continuous function $u:M\times N\rightarrow\mathbb{R}$ such that $\widetilde{\phi} = \phi - u + u\circ F$;
	
	\begin{proposition}
		\label{prophomolo}
		Let $\phi:M\times N\rightarrow \mathbb{R}$ be a H\"older continuous potential. There exists a H\"older continuous potential $\widetilde{\phi}:M\times N \rightarrow \mathbb{R}$ not depending on the stable direction such that:
		\begin{enumerate}
			\item [(1)] $\widetilde{\phi}$ is homologous to $\phi$;
			\item [(2)] if $\phi$ is zooming, then $\widetilde{\phi}$ is zooming;
			\item [(3)] $P_{F}(\widetilde{\phi})=P_{F}(\phi)$;
		\end{enumerate}
	\end{proposition}
	
	Before starting the proof of the proposition, let us state the following result that calculates the entropy of a skew product, considering the dynamics on the base and the dynamics on the fibers. 
	
	\begin{theorem}[Ledrappier-Walters Formula]
		\label{thLed-Walters}
		Let $\Hat{X}, X$ be compact metric spaces and let $\Hat{T}:\Hat{X}\longrightarrow \Hat{X}$, $T:X\longrightarrow X$ and $\Hat{\pi}:\Hat{X}\longrightarrow X$ be continuous maps such that $\Hat{\pi}$ is surjective and $\Hat{\pi}\circ \Hat{T} = T\circ \Hat{\pi}$. Then
		$$
		\sup_{\Hat{\nu}; \Hat{\pi}_{\ast}\Hat{\nu}=\nu}h_{\Hat{\nu}}(\Hat{T})= h_{\nu}(T) + \int h_{top}(\Hat{T}, \Hat{\pi}^{-1}(y)) d\nu(y).
		$$
	\end{theorem}
	
	Since $g(x, \cdot):N \to N$ is a uniform contraction, for every $x\in M$, we have that $h_{top}(F, \pi_1^{-1}(x))=0$ for every $x\in M$. Then, by Theorem \ref{thLed-Walters}, we obtain
	\begin{equation}
		\label{equahfHF}
		h_{\widetilde{\mu}}(F)=h_{\mu}(f)
	\end{equation}
	for every $\mu\in \mathcal{M}_{f}(M)$ and $\widetilde{\mu}\in \mathcal{M}_{F}(M\times N)$ such that $\pi_{1\ast}\widetilde{\mu}=\mu$. 
	
	\begin{proof}{Proposition \ref{prophomolo}.} 
		\begin{enumerate}
			\item [(1)] Let us consider,
			$$
			u(x,y)=\sum_{j=0}^{\infty}(\phi\circ F^{j}(x,y)-\phi\circ F^{j}(x,y_0))
			$$
			where $y_0\in N$ is the fixed point of the dynamics.  
			
			Considering the contraction of $F$ on the fibers, we have that $u$ is a continuous function. Defining the potential $\widetilde{\phi} = \phi - u + u\circ F$ we have that $\widetilde{\phi}$ is a continuous function homologous to $\phi$. Moreover, substituting $u$ into $\widetilde{\phi}$ and $g(x,y_0)=y_0$ 
			from ii.) \ref{conditionong}, we find that $\widetilde{\phi}$ does not depend on the stable direction.
			
			\item [(2)]  By definition of $\widetilde{\phi}$, and taking $\mu\in \mathcal{M}_{F}(M\times N)$,  then
			$$
			\int \phi d\mu = \int \widetilde{\phi} d\mu.
			$$
			Therefore, if $\phi$ is zooming then $\widetilde{\phi}$ is zooming. 
			
			\item [(3)] By definition of topological pressure and item (2) above. 
		\end{enumerate}
	\end{proof}
	On the other hand, since $\widetilde{\phi}$ does not depend on the stable direction, it induces a potential on the dynamics of the base. That is $\widetilde{\phi}(x,y_0)=\varphi(x)$ for all $x\in M$. We will consider these potentials for condition ii.) of \ref{conditionong}. For case i.) of \ref{conditionong}, we will consider potentials that are constant on the fibers. That is, given $\widehat{\varphi}\in C^{0}(M,\mathbb{R})$ define the function
	\[
	\begin{array}{cccc}
		\label{defitionvarphi}
		\widehat{\phi}\ : & \! M\times N & \! \longrightarrow
		& \! \mathbb{R} \\
		& \! (x,y) & \! \longmapsto
		& \! \widehat{\phi}(x,y):=\widehat{\varphi}(x).
	\end{array}
	\]
	We have that $\widehat{\phi}\in C^{0}(M\times N, \mathbb{R})$. 
	
	Now, let $\mathcal{M}^{1}_{\mu}(F, M\times N)$ be the set of all invariant probability measures $\widetilde{\mu}$ on $M\times N$ such that
	\[
	\pi_{1\ast}\widetilde{\mu} = \widetilde{\mu}\circ \pi^{-1}_{1}=\mu, 
	\]
	where $\mu \in \mathcal{M}_{f}(M)$ and $\pi_1:M \times N \rightarrow M$ stands for the first projection ($\pi_1(x,y)=x$). 

	We will use the following result (see \cite{PL}).
	
	

	\begin{lemma}
		\label{lemma-medidainvariante}
		Given $\lambda\in \mathcal{M}_{f}(M)$, then there exists $\mu\in \mathcal{M}_{F}(M\times N)$ such that $\pi_{1\ast}\mu = \lambda$.
	\end{lemma}
	
	\begin{proof}
		Invariant measures existence theorem. (see \cite{LUDWIG}[Theorem 1.5.10])
	\end{proof}
	\begin{remark}
		\label{remark -igualdad de integrales}
		Now consider the potential $\widehat{\phi}$, such that $\widehat{\phi}= \widehat{\varphi} \circ \pi_1$, where $\lambda$ is the equilibrium state for $\widehat{\varphi}$. Observe that, since $\pi_{1}\circ F = f \circ \pi_1$ and $\lambda = \pi_{1\ast}\mu$ are invariant measures, it holds
		\begin{equation}
			\int \widehat{\phi} d\mu = \int \widehat{\varphi} d\lambda.
		\end{equation}
		Similarly, we have that
		\begin{equation}
			\int \widetilde{\phi} d\mu = \int \varphi d\lambda
		\end{equation}
		where, $\widetilde{\phi}(x,y_0)=\varphi(x)$ for all $x\in M$.
	\end{remark}

	Therefore, by the remark (\ref{remark -igualdad de integrales}), the Ledrappier-Walters formula equation (\ref{equahfHF}) and the Proposition \ref{prophomolo} we get

		
		\begin{lemma}
			\label{lema2potenzooming}
			\begin{itemize}
				\item If $\phi$ is a zooming potential for $F$, then for $\widetilde{\phi}(x,y) = \varphi(x)$ is a zooming potential for $f$.
				
				\item If $\widehat{\varphi}$ is a zooming potential for $f$, then for $\widehat{\phi}(x,y) = \widehat{\varphi}(x)$ is a zooming potential for $F$.
			\end{itemize}
		\end{lemma}
		
		As a consequence, we have that:	
		\begin{align*}
			P_{f}(\widehat{\varphi})&=\sup_{\eta\in \mathcal{M}_{f}(M)} \{h_{\eta}(f) + \int \widehat{\varphi} d\eta  \}\\
			& = \sup_{\pi_{1\ast}\tilde{\eta}} \{ h_{\tilde{\eta}}(F) + \int \widehat{\phi} d\tilde{\eta}\} \leq P_{F}(\widehat{\phi}).
		\end{align*}
		
		Similarly, we prove that, 
		$$
		P_{F}(\widehat{\phi}) \leq P_{f}(\widehat{\varphi}),
		$$
		therefore, 
		\begin{equation}
			\label{igualdadpresion}
			P_{F}(\widehat{\phi}) = P_{f}(\widehat{\varphi}).  
		\end{equation}



		\begin{lemma}
			\label{lemaestadodeequlibri}
			If $\mu\in \mathcal{M}_{f}(M)$ is ergodic, then there exists a unique ergodic measure $\widetilde{\mu}\in \mathcal{M}_{F}(M\times N)$ such that $\mu=\pi_{1\ast}\widetilde{\mu}$. Moreover, $\mu$ is an equilibrium state for $(f,\varphi)$ if and only if $\widetilde{\mu}$ is an equilibrium state for $(F,\phi)$. 
		\end{lemma}
		\begin{proof} 
			The proof of this fact is due to the definition of topological pressure, Lemma \ref{lemma-medidainvariante}, and equation ( \ref{igualdadpresion}).  The uniqueness of $\tilde{\mu}$ is due to shrinkage along the fibers. See details in \cite{ARS}.
		\end{proof}
		
		Now, we will prove Theorem \ref{thB}. Given $(F',\phi')\in (\mathcal{FZ})'$, consider a sequence $(F_n,\widehat{\phi}_n)\in (\mathcal{FZ})'$ converging to $(F',\phi')$ in the $C^{0}-$topology. Let $\tilde{\mu}_n$ be an equilibrium state of $(F_n,\widehat{\phi}_n)$. We are going to show that every accumulation point $\tilde{\mu}$ of the sequence $(\tilde{\mu}_n)_n$ is an equilibrium state for $(F',\phi')$.
		
		For each $n\geq 1$, let $\widehat{\phi}_n$ be the potential associated to $\widehat{\varphi}_n$ be the potential on $M$. It follows from Lemma \ref{lema2potenzooming} that $\widehat{\varphi}_n$ is a zooming potential for $f_n$, and this means that $(f_n, \widehat{\varphi}_n)\in \mathcal{FZ}$ for all $n\geq 1$. Moreover, using the definitions of $g$ and $\widehat{\phi}(x,y)= \widehat{\varphi}(x) \ \forall \ y\in N$ above, we have that the $C^0-$converge of $(F_n, \widehat{\phi}_n)$ to $(F', \phi')$ implies the $C^0-$converge of $(f_n,\widehat{\varphi}_n)$ to $(f, \widehat{\varphi})$.
		
		For each $n\in \mathbb{N}$ consider $\mu_n=\pi_{1\ast}\tilde{\mu}_n$. From Lemma \ref{lemaestadodeequlibri} we have that $\mu_n$ is an equilibrium state for $(f_n, \widehat{\varphi}_n)$. Since the projection $\pi_1$ is continuous, if $\tilde{\mu}$
		is an accumulation point of the sequence $(\widetilde{\mu}_n)_n$, then $\mu=\pi_{1\ast}\tilde{\mu}$ is an accumulation point of $(\mu_n)_n$. By the equilibrium stability given by Theorem \ref{A}, we have that $\mu$ is an equilibrium state for $(f, \widehat{\varphi})$. Hence, applying Lemma \ref{lemaestadodeequlibri} again we have that $\tilde{\mu}$ is an equilibrium state for $(F', \phi')$. 	
		\begin{lemma}
			\label{lemcontinuidad}
			The function
			$$
			\begin{array}{cccc}
				& \! (\mathcal{FZ})' & \! \longrightarrow
				& \! \mathbb{R} \\
				& \! (F, \phi) & \! \longmapsto
				& \! P_{F}(\phi)
			\end{array}
			$$
			is continuous.
		\end{lemma}

		\begin{proof}
			Given $(F, \widehat{\phi})\in \mathcal{(FZ)}'$, consider the induced system $(f, \widehat{\varphi})\in \mathcal{FZ}$. Using $(\widehat{\phi}(x,y)= \widehat{\varphi}(x))$, it follows that if $(F,\widehat{\phi})$ varies continuously in $(\mathcal{FZ})'$, then $(f,\widehat{\varphi})$ also varies continuously in $\mathcal{FZ}$. It follows from equation (\ref{igualdadpresion}) that $P_f(\widehat{\varphi})=P_F(\widehat{\phi})$. Hence by Theorem \ref{A}, we have $P_f(\widehat{\varphi})$ varying continuously in the $C^0-$topology within $\mathcal{FZ}$, then $P_{F}(\widehat{\phi})$ varies continuously within $(\mathcal{FZ})'$ as well.
		\end{proof}
		
 \section{Proof of Theorem \ref{C}}		
		
		We divide the proof into two lemmas.
		
		\begin{lemma}\label{finite}
			Let $M$ be a compact space and $f: M \to M$ a continuous open zooming system. The set of equilibrium states of all continuous zooming potentials is finite.
		\end{lemma} 
		\begin{proof}
			We have finiteness of equilibrium states for every zooming H\"older potential $\phi : M \to \mathbb{R}$. Once $M$ is compact, by \cite{OV}[Theorem A.3.13], there exists a dense countable set $\mathcal{S} \subset \mathcal{C}^{0}(M)$, the space of continuous potentials. The space of continuous H\"older potentials $\mathcal{H}$ is dense. Given $\phi_{n} \in \mathcal{S}$ and $m \in \mathbb{N}$, there exists $\phi_{n}^{m} \in \mathcal{H}$ such that $\parallel \phi_{n} - \phi_{n}^{m} \parallel < 1/m$, which shows that the countable set $\mathcal{S}_{0} = \{\phi_{n}^{m}\}$ is dense in $\mathcal{H}$.
			
			We remind that the set $\mathcal{Z}$ of continuous zooming potentials is open. Since the space of H\"older potentials $\mathcal{H}$ is residual, the intersection $\mathcal{ZH} = \mathcal{Z} \cap \mathcal{S}_{0}$ is dense in $\mathcal{Z}$. Every potential in $\mathcal{ZH}$ has finitely many ergodic equilibrium states. So, the set of ergodic equilibrium states of potentials in $\mathcal{ZH}$ is countable. If necessary, we can add to $S_{0}$ the sequences $\varphi_{p,m,n}(x) = \phi_{n}^{m}(x) - (\psi(x) - \int \psi d\mu)$, as in Lemma \ref{sequence}, where $\psi$ can be taken H\"older. 
			
			By stability, this set of equilibrium states is closed with countable boundary, once the set of accumulation points of equilibrium states is also of equilibrium states. It means that this boundary is finite, since the boundary is a countable perfect set.
			
			But also by stability, the boundary coincide with the equilibrium states of a residual of H\"older potentials. Again by stability, all the continuous potentials has their ergodic equilibrim states among a finite list of ergodic zooming measures. 
		\end{proof}
		\begin{lemma}\label{sequence}
			Let $M$ be a compact metric space and $f: M \to M$ a continuous open zooming system. The set of continuous potentials with uniqueness is dense in the set of potentials with finiteness. As a consequence, given $\phi : M \to \mathbb{R}$ a continuous zooming potential whose induced potential is locally H\"older, we have uniqueness of equilibrium state.
		\end{lemma} 
		\begin{proof}
			We have finiteness of equilibrium states. Let $\mu_{1}, \dots, \mu_{k}$ be the ergodic equilibrium states.  We fix some $i \leq k$ and take a continuous potential $\psi_{i} : M \to \mathbb{R}$  such that $\int\psi_{i} d\mu_{i} < \int \psi_{i} d\mu_{j}, j \neq i$. Define for $a_{n} \downarrow 0$
			\[
			\varphi_{n}^{i}(x) = \phi(x) - a_{n}\bigg{(}\psi_{i}(x) - \int \psi_{i} d\mu_{i}\bigg{)}.
			\]
			Given $\eta \neq \mu_{j}$ for every $j \leq k$ it holds that $\eta$ is not an equilibrium state of $\varphi_{n}^{i}$ for infinitely many $n$. In fact, once $\eta \neq \mu_{j}$ if we could find $n_{1}, n_{2}, \dots$ such that $\eta$ is an equilibrium state of $\varphi_{n_{1}}^{i}, \varphi_{n_{2}}^{i}, \dots$, by stability we would have $\eta$ as an equilibrium state of $\phi$, which is a contradiction. Then, $\eta$ can only be an equilibrium state of $\varphi_{n}^{i}$ for infinitely many $n$ if $\eta = \mu_{j}$ for some $j \leq k$.
			
			If $\eta = \mu_{j}$ for some $j \neq i$, we obtain
			\[
			h_{\eta}(f) + \int \varphi_{n}^{i} d\eta = h_{\eta}(f) + \int \phi d\eta - a_{n}\int \bigg{(}\psi_{i}(x) - \int \psi_{i} d\mu_{i}\bigg{)} d \eta < 
			\]
			\[
			h_{\eta}(f) + \int \phi d\eta = h_{\mu_{i}}(f) + \int \phi d\mu_{i} = h_{\mu_{i}}(f) + \int \varphi_{n}^{i} d\mu_{i}.
			\]
			It implies that $\mu_{i}$ is the unique equilibrium state of $\varphi_{n}^{i}$ for infinitely many $n$. It shows that the set of potentials with uniqueness is dense.
			
			There exists a connected neighbourhood $\mathcal{U}$ of the potential $\phi$ such that every potential $\varphi \in \mathcal{U}$ with uniqueness has some of the measures  $\mu_{1}, \dots, \mu_{k}$ as its unique equilibrium state. Let $\mathcal{U}_{i}$ be the set of potentials in $\mathcal{U}$ with $\mu_{i}$ as the unique equilibrium state. By stability, we have that every  $\mathcal{U}_{i}$ is both open and closed. By connectness, we have that $\mathcal{U} = \mathcal{U}_{i}$ for every $i \leq k$. It means that the potential $\phi$ has uniqueness.
			
			We can construct the neighbourhood $\mathcal{U}$ as following. Let $\eta$ be an ergodic measure which is not an equilibrium state of $\phi$. Denote by $\mathcal{U}_{\eta}$ the set of continuous potentials such that $\eta$ is the unique equilibrium state. We have $\phi \not \in \overline{\mathcal{U}_{\eta}}$. By Lemma \ref{finite} there are finitely many equilibrium states at all $\mu_{1}, \dots, \mu_{k}, \nu_{1}, \dots, \nu_{p}$ such that $\nu_{j}$ is not equilibrium state of $\phi$ for every $j \leq p$. So, we can take $\mathcal{U} \subset \mathcal{V}$ where
			\[
		    \mathcal{V} = \bigg{(}\bigcup_{j=1}^{p} \overline{\mathcal{U}_{\nu_{j}}}\bigg{)}^{c}.
			\]
			The uniqueness is established and the lemma proved.
		\end{proof}

		\section{Applications}
		
		In this section, we give examples of  zooming systems. All the examples are given in \cite{S1} and we reproduce them here.
		
		\subsection{Viana maps} We recall the definition of the open class of maps with critical sets in dimension 2, introduced by M. Viana in \cite{V}. We skip the technical
		points. It can be generalized for any dimension (See \cite{A}).
		
		Let $a_{0} \in (1,2)$ be such that the critical point $x=0$ is pre-periodic for the quadratic map $Q(x)=a_{0} - x^{2}$. Let $S^{1}=\mathbb{R}/\mathbb{Z}$ and 
		$b:S^{1} \to \mathbb{R}$ a Morse function, for instance $b(\theta) = \sin(2\pi\theta)$. For fixed small $\alpha > 0$, consider the map
		\[
		\begin{array}{c}
			f_{0}: S^{1} \times \mathbb{R} \longrightarrow S^{1} \times \mathbb{R}\\
			\,\,\,\,\,\,\,\,\,\,\,\,\,\,\,\,\,\,\,\ (\theta,x) \longmapsto (g(\theta),q(\theta,x))
		\end{array}
		\] 
		where $g$ is the uniformly expanding map of the circle defined by $g(\theta)=d\theta
		(mod\mathbb{Z})$ for some $d \geq 16$, and $q(\theta,x) = a(\theta) - x^{2}$ with $a(\theta) = a_{0} + \alpha b(\theta)$. It is easy to check that for $\alpha > 0$ 
		small enough there is an interval $I \subset (-2,2)$ for which $f_{0}(S^{1} \times I)$ is contained in the interior of $S^{1} \times I$. Thus, any map $f$ sufficiently
		close to $f_{0}$ in the $C^{0}$ topology has $S^{1} \times I$ as a forward invariant region. We consider from here on these maps $f$ close to $f_{0}$ restricted to 
		$S^{1} \times I$. Taking into account the expression of $f_{0}$ it is not difficult to check that for $f_{0}$ (and any map $f$ close to $f_{0}$ in the $C^{2}$ topology)
		the critical set is non-degenerate.
		
		The main properties of $f$ in a $C^{3}$ neighbourhood of $f$ that we will use here are summarized below (See \cite{A},\cite{AV},\cite{Pi1}):
		
		\begin{enumerate}
			\item[(1)] $f$ is \textbf{\textit{non-uniformly expanding}}, that is, there exist $\lambda > 0$ and a Lebesgue full measure set $H \subset S^{1} \times I$ such that 
			for every point $p=(\theta, x) \in H$, the following holds
			\[
			\displaystyle \limsup_{n \to \infty} \frac{1}{n} \sum_{i=0}^{n-1} \log \parallel Df(f^{i}(p))^{-1}\parallel^{-1} < -\lambda.
			\]  
			\item[(2)] Its orbits have \textbf{\textit{slow approximation to the critical set}}, that is, for every $\epsilon > 0$ the exists $\delta > 0$ such that for every point
			$p=(\theta, x) \in H \subset S^{1} \times I$, the following holds 
			\[
			\displaystyle \limsup_{n \to \infty} \frac{1}{n} \sum_{i=0}^{n-1} - \log \text{dist}_{\delta}(p,\mathcal{C}) < \epsilon.
			\]  
			where 
			\[
			\text{dist}_{\delta}(p,\mathcal{C}) =  
			\left\{ 
			\begin{array}{ccc}
				dist(p,\mathcal{C}), & if & dist(p,\mathcal{C}) < \delta\\
				1 & if & dist(p,\mathcal{C}) \geq \delta 
			\end{array}
			\right.
			\]  
			\item[(3)] $f$ is topologically mixing;
			
			\item[(4)] $f$ is strongly topologically transitive;
			
			\item[(5)] it has a unique ergodic absolutely continuous invariant (thus SRB) measure;
			
			\item[(6)]the density of the SRB measure varies continuously in the $L^{1}$ norm with $f$.
		\end{enumerate}
		
		\begin{remark}
			We observe that this definition of non-uniformly expansion is included in ours by neighbourhoods.
		\end{remark}
		
		\subsection{Benedicks-Carleson Maps} We study a class of non-hyperbolic maps of the interval with the condition of exponential growth of the derivative at critical values, called 
		\textbf{\textit{Collet-Eckmann Condition}}. We also ask the map to be $C^{2}$ and topologically mixing and the critical points to have critical order 
		$2 \leq \alpha < \infty$.
		
		Given a critical point $c \in I$, the \textbf{\textit{critical order}} of $c$ is a number $\alpha_{c} > 0$ such that 
		$f(x) = f(c) \pm |g_{c}(x)| ^{\alpha_{c}}, \,\, \text{for all} \, \, x \in \mathcal{U}_{c}$ where $g_{c}$ is a diffeomorphism 
		$g_{c}: \mathcal{U}_{c} \to g(\mathcal{U}_{c})$ and $\mathcal{U}_{c}$ is a neighbourhood of $c$. 
		
		Let $\delta>0$ and denote $\mathcal{C}$ the set of critical points and $\displaystyle B_{\delta} = \cup_{c \in \mathcal{C}} (c - \delta, c + \delta)$. 
		Given $x \in I$, we suppose that
		
		\begin{itemize}
			\item \textbf{(Expansion outside $B_{\delta}$)}.  There exists $\kappa > 1 $ and $\beta > 0$ such that, if $x_{k} = f^{k}(x) \not \in B_{\delta}, \,\, 0 \leq k \leq n-1$ then $|Df^{n}(x)| \geq \kappa \delta^{(\alpha_{\max} -1)}e^{\beta n}$, where $\alpha_{\max} = \max \{\alpha_{c}, c \in \mathcal{C}\}$. Moreover, if $x_{0} \in f(B_{\delta})$ or $x_{n} \in B_{\delta}$ then $|Df^{n}(x)| \geq \kappa e^{\beta n}$.
			
			\item \textbf{(Collet-Eckmann Condition)}. There exists $\lambda > 0$ such that 
			\[
			|Df^{n}(f(c))| \geq e^{\lambda n}.
			\]
			
			\item \textbf{(Slow Recurrence to $\mathcal{C}$)}. There exists $\sigma \in (0, \lambda/5)$ such that 
			\[
			dist(f^{k}(x), \mathcal{C}) \geq e^{-\sigma k}.
			\]
		\end{itemize}
		
		\subsection{Rovella Maps}
		
		There is a class of non-uniformly expanding maps known as \textbf{\textit{Rovella Maps}}. They are derived from the so-called \textit{Rovella Attractor},
		a variation of the \textit{Lorenz Attractor}. We proceed with a brief presentation. See \cite{AS} for details.
		
		\subsubsection{Contracting Lorenz Attractor}
		
		The geometric Lorenz attractor is the first example of a robust attractor for a flow containing a hyperbolic singularity. The attractor is a transitive maximal invariant
		set for a flow in three-dimensional space induced by a vector field having a singularity at the origin for which the derivative of the vector field at the singularity has
		real eigenvalues $\lambda_{2} < \lambda_{3} < 0 < \lambda_{1}$ with $\lambda_{1} + \lambda_{3} > 0$. The singularity is accumulated by regular orbits which prevent the 
		attractor from being hyperbolic.
		
		The geometric construction of the contracting Lorenz attractor (Rovella attractor) is the same as the geometric Lorenz attractor. The only difference is the condition
		(A1)(i) below that gives in particular $\lambda_{1} + \lambda_{3} < 0$. The initial smooth vector field $X_{0}$ in $\mathbb{R}^{3}$ has the following properties:
		
		\begin{itemize}
			
			\item[(A1)] $X_{0}$ has a singularity at $0$ for which the eigenvalues $\lambda_{1},\lambda_{2},\lambda_{3} \in \mathbb{R}$ of $DX_{0}(0)$ satisfy:
			\begin{itemize}
				
				\item[(i)] $0 < \lambda_{1} < -\lambda_{3}  < -\lambda_{2}$,
				
				\item[(ii)] $r > s+3$, where $r=-\lambda_{2}/\lambda_{1}, s=-\lambda_{3}/\lambda_{1}$;
			\end{itemize}
			
			\item[(A2)] there is an open set $U \subset \mathbb{R}^{3}$, which is forward invariant under the flow, containing the cube
			$\{(x,y,z) : \mid x \mid \leq 1, \mid y \mid \leq 1, \mid x \mid \leq 1\}$ and supporting the \textit{Rovella attractor}
			\[
			\displaystyle \Lambda_{0} = \bigcap_{t \geq 0} X_{0}^{t}(U).
			\]
			
			The top of the cube is a Poincar\'e section foliated by stable lines $\{x = \text{const}\} \cap \Sigma$ which are invariant under Poincar\'e first return map $P_{0}$.
			The invariance of this foliation uniquely defines a one-dimensional map $f_{0} : I \backslash \{0\} \to I$ for which
			\[
			f_{0} \circ \pi = \pi \circ P_{0},
			\]
			where $I$ is the interval $[-1,1]$ and $\pi$ is the canonical projection $(x,y,z) \mapsto x$;
			
			\item[(A3)] there is a small number $\rho >0$ such that the contraction along the invariant foliation of lines $x =$const in $U$ is stronger than $\rho$.
		\end{itemize}
		
		See \cite{AS} for properties of the map $f_{0}$.
		
		\subsubsection{Rovella Parameters}
		
		The Rovella attractor is not robust. However, the chaotic attractor persists in a measure theoretical sense: there exists a one-parameter family of positive Lebesgue measure
		of $C^{3}$ close vector fields to $X_{0}$ which have a transitive non-hyperbolic attractor. In the proof of that result, Rovella showed that there is a set of parameters
		$E \subset (0,a_{0})$ (that we call \textit{Rovella parameters}) with $a_{0}$ close to $0$ and $0$ a full density point of $E$, i.e.
		\[
		\displaystyle \lim_{a \to 0} \frac{\mid E \cap (0,a) \mid}{a} = 1,
		\]
		such that:
		
		\begin{itemize}
			\item[(C1)] there is $K_{1}, K_{2} > 0$ such that for all $a \in E$ and $x \in I$
			\[
			K_{2} \mid x \mid^{s-1} \leq f_{a}'(x) \leq K_{1} \mid x \mid^{s-1},
			\]
			where $s=s(a)$. To simplify, we shall assume $s$ fixed.
			
			\item[(C2)] there is $\lambda_{c} > 1$ such that for all $a \in E$, the points $1$ and $-1$ have \textit{Lyapunov exponents} greater than $\lambda_{c}$:
			\[
			(f_{a}^{n})'(\pm 1) > \lambda_{c}^{n}, \,\, \text{for all} \, \, n \geq 0;
			\]
			
			\item[(C3)] there is $\alpha > 0$ such that for all $a \in E$ the \textit{basic assumption} holds:
			\[
			\mid f_{a}^{n-1}(\pm 1)\mid > e^{-\alpha n}, \,\, \text{for all} \, \, n \geq 1;
			\]
			
			\item[(C4)] the forward orbits of the points $\pm 1$ under $f_{a}$ are dense in $[-1,1]$ for all $a \in E$.
		\end{itemize}
		
		\begin{definition}
			We say that a map $f_{a}$ with $a \in E$ is a \textbf{\textit{Rovella Map}}. 
		\end{definition}
		
		\begin{theorem}
			(Alves-Soufi \cite{AS}) Every Rovella map is non-uniformly expanding. 
		\end{theorem}
		
		\subsection{Hyperbolic Times}
		
		The idea of hyperbolic times is a key notion on the study of non-uniformly hyperbolic dynamics and it was introduced by Alves et al. 
		This is powerful to get expansion in the context of non-uniform expansion. Here, we recall the basic definitions and results on hyperbolic times that we will use later on. 
		We will see that this notion is an example of a Zooming Time. 
		
		In the following, we give definitions taken from \cite{A} and \cite{Pi1}.
		
		\begin{definition}
			Let $M$ be a compact Riemannian manifold of dimension $d \geq 1$ and $f:M \to M$ a continuous map defined on $M$.
			The map $f$ is called \textbf{\textit{non-flat}} if it is a local $C^{1 + \alpha}, (\alpha >0)$ diffeomorphism in the whole manifold except in a 
			non-degenerate set $\mathcal{C} \subset M$. We say that $M \neq \mathcal{C} \subset M$ is a \textbf{\textit{non-degenerate set}}
			if there exist $\beta, B > 0$ such that the following two conditions hold.
			
			\begin{itemize}
				\item $\frac{1}{B} d(x,\mathcal{C})^{\beta} \leq \frac{\parallel Df(x) v\parallel}{\parallel v \parallel} \leq B d(x,\mathcal{C})^{-\beta}$ for all $v \in T_{x}M$, for every $x \in M\backslash\mathcal{C}$.
				
				For every $x, y \in M \backslash \mathcal{C}$ with $d(x,y) < d(x,\mathcal{C})/2$ we have
				
				\item $\mid \log \parallel Df(x)^{-1} \parallel - \log \parallel Df(y)^{-1} \parallel \mid \leq \frac{B}{d(x,\mathcal{C})^{\beta}} d(x,y)$.
			\end{itemize}
			
		\end{definition}
		
		In the following, we give the definition of a hyperbolic time \cite{ALP2}, \cite{Pi1}.
		
		\begin{definition}
			(Hyperbolic times). Let us fix $0 < b = \frac{1}{3} \min\{1,1 \slash \beta\} < \frac{1}{2} \min\{1,1\slash \beta\}$. 
			Given $0 < \sigma < 1$ and $\epsilon > 0$, we will say that $n$ is a $(\sigma, \epsilon)$\textbf{\textit{-hyperbolic time}} for a point $x \in M$ 
			(with respect to the non-flat map $f$ with a $\beta$-non-degenerate critical/singular set $\mathcal{C})$ if for all $1 \leq k \leq n$ we have 
			\[
			\prod_{j=n-k}^{n-1} \|(Df \circ f^{j}(x))^{-1}\| \leq \sigma^{k} \,\, \text{and} \,\, dist_{\epsilon}(f^{n-k}(x), \mathcal{C}) \geq \sigma^{bk}.
			\]
			where
			\[
			\text{dist}_{\epsilon}(p,\mathcal{C}) =  
			\left\{ 
			\begin{array}{ccc}
				dist(p,\mathcal{C}), & if & dist(p,\mathcal{C}) < \epsilon\\
				1 & if & dist(p,\mathcal{C}) \geq \epsilon. 
			\end{array}
			\right.
			\]
			We denote de set of points of $M$ such that $n \in \mathbb{N}$ is a $(\sigma,\epsilon)$-hyperbolic time by $H_{n}(\sigma,\epsilon,f)$.
		\end{definition}
		
		\begin{proposition}
			(Positive frequence). Given $\lambda > 0$ there exist $\theta > 0$ and $\epsilon_{0} > 0$ such that, for every $x \in M$ and $\epsilon \in (0,\epsilon_{0}]$,
			\[
			\#\{1 \leq j \leq n mid \,\, x \in H_{j}(e^{-\lambda \slash 4}, \epsilon, f) \} \geq \theta n,
			\]
			whenever $\frac{1}{n}\sum_{i=0}^{n-1}\log\|(Df(f^{i}(x)))^{-1}\|^{-1} \geq \lambda$ and $\frac{1}{n}\sum_{i=0}^{n-1}-\log dist_{\epsilon}(x, \mathcal{C}) \leq \frac{\lambda}{16 \beta}$.
		\end{proposition}
		
		Denote by $\mathcal{H}$ the set of point $x \in M$ such that
		\[
		\displaystyle \limsup_{n \to \infty} \frac{1}{n} \sum_{i=0}^{n-1} \log \parallel Df(f^{i}(p))^{-1}\parallel^{-1} < -\lambda.
		\] 
		and
		\[
		\displaystyle \limsup_{n \to \infty} \frac{1}{n} \sum_{i=0}^{n-1} - \log \text{dist}_{\delta}(p,\mathcal{C}) < \epsilon.
		\]
		If $f$ is non-uniformly expanding, it follows from the proposition that the points of $\mathcal{H}$ have infinitely many moments with positive frequency of hyperbolic times. In particular, they have infinitely many hyperbolic times.
		
		The following proposition shows that the hyperbolic times are indeed zooming times, where the zooming contraction is $\alpha_{k}(r) = \sigma^{k/2}r$.
		
		\begin{proposition}
			Given $\sigma \in (0,1)$ and $\epsilon > 0$, there is $\delta,\rho > 0$, depending only on $\sigma$ and $\epsilon$ and on the map $f$, such that if $x \in H_{n}(\sigma,\epsilon,f)$ then there exists a neighbourhood $V_{n}(x)$ of $x$ with the following properties:
			
			\begin{enumerate}
				\item[(1)] $f^{n}$ maps $\overline{V_{n}(x)}$ diffeomorphically onto the ball $\overline{B_{\delta}(f^{n}(x))}$;
				\item[(2)] $dist(f^{n-j}(y),f^{n-j}(z)) \leq \sigma^{j\slash 2} dist(f^{n}(y), f^{n}(z)), \text{for all} \, \, y,z \in V_ {n}(x)$ and $1 \leq j < n$.
				\item[(3)]$\log \frac{\mid \det Df^{n}(y)\mid}{\mid \det Df^{n}(z)\mid} \leq \rho d(f^{n}(y),f^{n}(z))$.
			\end{enumerate}
			
			for all $y,z \in V_{n}(x)$.
		\end{proposition}
		
		The sets $V_{n}(x)$ are called hyperbolic pre-balls and their images $f^{n}(V_{n}(x)) = B_{\delta}(f^{n}(x))$, hyperbolic balls.
		
		\bigskip
		
		In the following, we give definitions for a map on a metric space to have similar behaviour to maps with hyperbolic times and which can be found in \cite{Pi1}.  
		
		Given $M$ a metric spaces and $f: M \to M$, we define for $p \in M$:
		\[
		\displaystyle \mathbb{D}^{-}(p) = \liminf_{x \to p} \frac{d(f(x),f(p)}{d(x,p)}
		\]
		Define also,
		\[
		\displaystyle \mathbb{D}^{+}(p) = \limsup_{x \to p} \frac{d(f(x),f(p)}{d(x,p)}
		\]
		We will consider points $x \in M$ such that 
		\[
		\displaystyle \limsup_{n \to \infty} \frac{1}{n} \sum_{i=0}^{n-1} \log \mathbb{D}^{-} \circ f^{i}(x) > 0.
		\]  
		The critical set $\mathcal{C}$ is the set of points $x \in M$ such that $\mathbb{D}^{-}(x) = 0$ or $\mathbb{D}^{+}(x) = \infty$. For  the non-degenerateness we ask that $\mathcal{C} \neq M$ and there exist $B, \beta >0$ such that
		
		\begin{itemize}
			
			\item $\frac{1}{B} d(x,\mathcal{C})^{\beta} \leq \mathbb{D}^{-}(x) \leq \mathbb{D}^{+}(x) \leq B d(x,\mathcal{C})^{-\beta}, x \not \in \mathcal{C}$.
			
			For every $x, y \in M \backslash \mathcal{C}$ with $d(x,y) < d(x,\mathcal{C})/2$ we have
			
			\item $\mid \log \mathbb{D}^{-}(x) - \log \mathbb{D}^{-}(y) \mid \leq \frac{B}{d(x,\mathcal{C})^{\beta}} d(x,y)$.
			
		\end{itemize}
		
		With these conditions we can see that all the consequences for hyperbolic times are valid here and the expanding sets and measures are zooming sets and measures.
		
		\begin{definition}
			We say that a map is \emph{conformal at p} if $\mathbb{D}^{-}(p) = \mathbb{D}^{+}(p)$. So, we define
			\[
			\displaystyle \mathbb{D}(p) = \lim_{x \to p} \frac{d(f(x),f(p)}{d(x,p)}.
			\]
		\end{definition}
		
		Now, we give an example of such an open non-uniformly expanding map.  
		
		\subsection{Expanding sets on a metric space} Let $\sigma : \Sigma_{2}^{+} \to \Sigma_{2}^{+}$ be the one-sided shift, with the usual metric:
		\[
		\displaystyle d(x,y) = \sum_{n=1}^{\infty} \frac{\mid x_{n} - y_{n} \mid}{2^{n}},
		\]
		where $x = \{x_{n}\}, y = \{y_{n}\}$. We have that $\sigma$ is a conformal map such that $\mathbb{D}^{-}(x) = 2, \text{for all} \, \, \, x \in \Sigma_{2}^{+}$. Also, every forward invariant set (in particular the whole $\Sigma_{2}^{+}$)  and all invariant measures for the shift $\sigma$ are expanding (then they are zooming). In particular, if we consider an invariant set that is not dense such that the reference measure has a Jacobian with bounded distortion, we can obtain an open shift map with $H \neq \emptyset$. To be precise, by taking any (previously fixed) zooming set $\Lambda \subset \Sigma_{2}^{+}$ which is not dense such that the reference measure has a Jacobian with bounded distortion, we apply our Theorem \ref{A} to obtain an open zooming system and a Markov structure adapted to a hole $H \subset \Sigma_{2}^{+}$ such that $H \cap \Lambda = \emptyset$. It is enough to take $r_{0}>0$ such that one of the balls of the open cover is disjoint from $\Lambda$. We can obtain finiteness of equilibrium states.
		
		\subsection{Zooming sets on a metric space (not expanding)} Let $\sigma : \Sigma_{2}^{+} \to \Sigma_{2}^{+}$ be the one-sided shift, with the following metric for $\sum_{n=1}^{\infty} b_{n} < \infty$:
		\[
		\displaystyle d(x,y) = \sum_{n=1}^{\infty} b_{n}\mid x_{n} - y_{n} \mid,
		\]
		where $x = \{x_{n}\}, y = \{y_{n}\}$ and $b_{n+k} \leq b_{n}b_{k}$ for all $n,k \geq 1$. By induction, it means that $b_{n} \leq b_{1}^{n}$. Let us suppose that $b_{n} \leq a_{n}:=(n+b)^{-a}, a>1, b>0$ for all $n \geq 1$. 
		
		We claim that $a_{n}$ defines a Lipschitz contraction for the shift map. We require that there exists $n_{0} > 1$ such that $b_{n} > a_{1}^{n} \geq b_{1}^{n}$ for $n \leq n_{0}$. So, the contraction is not exponential. In fact, if $x,y$ belongs to the cylinder $C_{k}$ we have
		\begin{eqnarray*}
			\displaystyle d(x,y) &=& \sum_{n=1}^{\infty} b_{n}\mid x_{n} - y_{n} \mid = \sum_{n=k+1}^{\infty} b_{n}\mid x_{n} - y_{n} \mid = \sum_{n=1}^{\infty} b_{n+k}\mid x_{n+k} - y_{n+k} \mid\\
			&\leq& b_{k} \sum_{n=1}^{\infty} b_{n}\mid x_{n+k} - y_{n+k} \mid = b_{k} d(\sigma^{k}(x),\sigma^{k}(y)) \leq a_{k} d(\sigma^{k}(x),\sigma^{k}(y)).
		\end{eqnarray*}
		It implies that
		\begin{eqnarray*}
			\displaystyle d(\sigma^{i}(x),\sigma^{i}(y)) \leq  a_{k-i} d(\sigma^{k-i}(\sigma^{i}(x)),\sigma^{k-i}(\sigma^{i}(y)))= a_{k-i}d(\sigma^{k}(x),\sigma^{k}(y)), i \leq k.
		\end{eqnarray*}
		It means that the sequence $a_{n}$ defines a Lipschitz contraction, as we claimed.
		
		Now, every forward invariant set (in particular the whole $\Sigma_{2}^{+}$)  and all invariant measures for the shift $\sigma$ are not expanding but they are  zooming. In particular, if we consider an invariant set that is not dense such that the reference measure has a Jacobian with bounded distortion, we can obtain an open shift map with $H \neq \emptyset$. To be precise, by taking any (previously fixed) zooming set $\Lambda \subset \Sigma_{2}^{+}$ which is not dense such that the reference measure has a Jacobian with bounded distortion, we apply our Theorem \ref{A} to obtain an open zooming system and a Markov structure adapted to a hole $H \subset \Sigma_{2}^{+}$ such that $H \cap \Lambda = \emptyset$. It is enough to take $r_{0}>0$ such that one of the balls of the open cover is disjoint from $\Lambda$. We can obtain finiteness of equilibrium states.

		\subsection{Uniformly expanding maps} As can be seen in \cite{OV} Chapter 11, we have the so-called \textbf{\textit{uniformly expanding maps}} which is defined on a compact differentiable manifold $M$ as a $C^{1}$ map $f:M \to M$ (with no critical set) for which there exists $\sigma > 1$ such that
		\[
		\|Df(x)v\|\geq \sigma \|v\|, \,\, \text{for every} \,\, x \in M, v \in T_{x}M.
		\]
		
		For compact metric spaces $(M,d)$ we define it as a continuous map $f:M \to M$, for which there exists $\sigma > 1, \delta>0$ such that for every $x \in M$ we have that the image of the ball $B(x,\delta)$ contains a neighbourhood of the ball $B(f(x),\delta)$ and
		\[
		d(f(a),f(b)) \geq \sigma d(a,b), \,\, \text{for every} \,\, a,b \in B(x,\delta).
		\]
		We observe that the uniformly expanding maps on differentiable manifolds satisfy the conditions for the definition on compact metric spaces, when they are seen as Riemannian manifolds.
		
		\subsection{Local diffeomorphisms}\label{local} As can be seen in details in \cite{A}, we will briefly describe a class of non-uniformly expanding maps.
		
		Here we present a robust ($C^{1}$ open) classes of local diffeomorphisms (with no critical set) that are non-uniformly expanding. Such classes of maps can be obtained, e.g., through deformation of a uniformly expanding map by isotopy inside some small region. In general, these maps are not uniformly expanding: deformation can be made in such way that the new map has periodic saddles.
		
		Let $M$ be a compact manifold supporting some uniformly expanding map $f_{0}$. $M$ could be the $d$-dimensional torus $\mathbb{T}^{d}$, for instance. Let $V \subset M$ be some small compact domain, so that the restriction of $f_{0}$ to $V$ is injective. Let $f$ be any map in a sufficiently small $C^{1}$-neighbourhood $\mathcal{N}$ of $f_{0}$ so that:
		
		\begin{itemize}
			\item $f$ is \textit{volume expanding everywhere}: there exists $\sigma_{1} > 1$ such that
			\[
			|\det Df(x)| > \sigma_{1} \,\, \text{for every} \,\, x \in M;
			\]	
			
			\item $f$ is \textit{expanding outside} $V$: there exists $\sigma_{0} > 1$ such that
			\[
			\|Df(x)^{-1}\| < \sigma_{0} \,\, \text{for every} \,\, x \in M \backslash V;
			\] 
			
			\item $f$ is \textit{not too contracting on} $V$: there is some small $\delta > 0$ such that
			\[
			\|Df(x)^{-1}\| < 1 + \delta \,\, \text{for every} \,\, x \in V.
			\]
		\end{itemize}
		
		In \cite{A} it is shown that this class satisfy the condition for non-uniform expansion. In the following we show a Lemma from \cite{A} which proves that such maps are non-uniformly expanding with Lebesgue as a reference measure.
		\begin{lemma}
			Let $B_{1},\dots, B_{k},B_{k+1}= V$ a partition of $M$ into domains such that $f$ is injective on $B_{j}, 1 \leq j \leq p+1$. There exists $\theta > 0$  such that the orbit of Lebesgue almost every point $x \in M$ spends a fraction $\theta$ of the time in $B_{1} \cap \dots B_{p}$, that is,
			\[
			\#\{0\leq j < n \mid f^{j}(x) \in B_{1} \cap \dots B_{p}\} \geq \theta n,
			\]
			for every large $n \in \mathbb{N}$.
		\end{lemma}
		
		\subsection{Open zooming systems from local diffeomorphisms} We can obtain an open zooming system such that the zooming set $\Lambda$ is disjoint from the hole $H$ ($\Lambda \cap H = \emptyset$) using a local diffeomorphism $f : M \to M$ which is non-uniformly expanding as in the subsection \ref{local}. 
		
		Let the zooming set $\Lambda = \cap_{j=-\infty}^{\infty}f^{j}(M\backslash V)$ with positive Lebesgue measure $m$. Since $\Lambda \cap V = \emptyset$, we can take a zooming reference measure $\mu = m / m(\Lambda)$ which has a Jacobian with bounded distortion. The zooming set $\Lambda$ is disjoint from $V$ and we can take  the hole $H \subset V$ given by Theorem \ref{A} (and $\Lambda \cap H = \emptyset$). This setup now allows us to obtain existence and finiteness of (open) equilibrium state. We observe that in the work \cite{ALP2}[Lemma 2.1](3) we have that the Lebesgue measure has a Jacobian with bounded distortion.
		
		As a concrete example on the interval $[a,b]$, we can take a dynamically defined Cantor set with positive Lebesgue measure. It can be seen in \cite{PT}[Chapter 4] as an expanding map $g$ over a disjoint union of intervals $P = I_{1} \uplus \dots \uplus I_{k}$ onto the interval $[a,b]$, where $I_{j} \in [a,b]$ is a compact interval for every $1 \leq j \leq k$, that is, every interval $I_{j}$ is taken onto $[a,b]$. Outside the union $P$ we can define the map to have a measurable map $f:[a,b] \to [a,b]$. In \cite{PT} we see that the map $g$ has bounded distortion and we can extend it to the map $f$ preserving this property. The hole $H$ can be taken outside the union $P$.

	\end{document}